\documentclass[12pt]{article}
\usepackage{amsmath,amsfonts,amsthm,amssymb,ytableau,multirow}

\usepackage[margin=1in]{geometry}
\usepackage{url}
\usepackage{graphicx}
\usepackage{float}
\usepackage{verbatim}

\newtheorem{thm}{Theorem}
\newtheorem{lem}[thm]{Lemma}

\newtheorem{dfn}[thm]{Definition}
\numberwithin{equation}{section}
\newdimen\Squaresize \Squaresize=14pt
\newdimen\Thickness \Thickness=0.4pt

\def\Square#1{\hbox{\vrule width \Thickness
   \vbox to \Squaresize{\hrule height \Thickness\vss
      \hbox to \Squaresize{\hss#1\hss}
   \vss\hrule height\Thickness}
\unskip\vrule width \Thickness} \kern-\Thickness}

\def\Vsquare#1{\vbox{\Square{$#1$}}\kern-\Thickness}

\title{Sums of products of binomial coefficients mod $2$
  and $2$-regular sequences}
\author{Narad Rampersad and Max Wiebe\footnote{
Department of Math/Stats,
University of Winnipeg,
515 Portage Ave.,
Winnipeg, MB, R3B 2E9
Canada; {\tt narad.rampersad@gmail.com}.}}

\begin{document}
\maketitle
\begin{abstract}
  Wu showed that certain sums of products of binomial coefficients
  modulo $2$ are given by the run length transforms of several famous
  linear recurrence sequences, such as the positive integers, the
  Fibonacci numbers, the extended Lucas numbers, and Narayana's cows
  sequence.  In this paper we show that the run length transform of
  such sequences are $2$-regular sequences.  This allows us to obtain
  Wu's results and some new ones using the computer program \texttt{Walnut},
  eliminating the need for long technical proofs.
\end{abstract}

\section{Introduction}
Wu \cite{Wu22} recently studied sums of the form
\begin{equation} \label{T(n)}
    T(n) = \sum^n_{k=0} \left[\dbinom{a_1n+a_2k}{a_3n+a_4k}\dbinom{n}{k}\pmod{2}\right],
\end{equation}
where $a_1+a_2 \geq 0$ and $a_3+a_4 \geq 0$.  He showed that for certain values of $a_1, a_2, a_3, a_4$,
the sequence defined by \eqref{T(n)} can be obtained as the run length transform of
a famous linear recurrence sequence, such as the positive integers, the
Fibonacci numbers, the extended Lucas numbers, or Narayana's cows
sequence.

The run length transform is a operation on integer sequences first introduced by
Sloane~\cite{Slo18}.
\begin{dfn} \label{Dfn2}
    The \emph{run length transform sequence} $(T(n))_{n\geq0}$ of a sequence $(S(n))_{n\geq0}$ is given by:
    \begin{equation} \label{dfn2eq1}
        T(n) = \sum_{i \in \mathcal{L}(n)}S(i)
    \end{equation}
    where $\mathcal{L}(n)$ is the list of the lengths of all maximal runs of $1$'s (with repetitions) in $[n]_2$,
    the binary representation of $n$. 
\end{dfn}
For example, if $n = 11$, then $[n]_2 = 1011$. So $\mathcal{L}(11)$ = $\{1, 2\}$ and $T(11) = S(1)S(2)$. 

Although they did not not state it this way, Sloane~\cite{Slo18} and Wu~\cite{Wu22}
showed that if $(S(n))_{n\geq0}$ is a linear recurrence sequence, then its run length transform
$(T(n))_{n\geq0}$ is a $2$-regular sequence.  This is a class of sequences with a deep theory
(see \cite{AS03}).  Furthermore, the computer package \texttt{Walnut} can be used to perform various computations
involving these sequences.  Our goal is to show how to use \texttt{Walnut} to obtain the results of Wu,
as well as some new ones.

\section{$2$-regular sequences and \texttt{Walnut}}\label{sec1}
Next we define the class of $2$-regular sequences.  We will give two equivalent definitions.
The first is the one used (implicitly) by Wu~\cite{Wu22} and the second is the one we will use
in the rest of this paper.

Let $(a(n))_{n \geq 0}$ be an integer sequence.  We define the \emph{$2$-kernel} of $a$ to be
the following set of subsequences:
\[
\mathcal{K}_2(a) = \{(a(2^in+j))_{n \geq 0} \;:\; i \geq 0;\;0 \leq j < 2^i\}.
\]
If there is a finite subset $R \subseteq \mathcal{K}_k(a)$ such that every sequence in
$\mathcal{K}_k(a)$ can be written as a linear combination over $\mathbb{Z}$ of
sequences in $R$, then $a$ is a \emph{$2$-regular sequence}.

For explicit calculation, the following equivalent definition may be more useful.
Consider a triple $(v,\gamma,w)$, where
\begin{itemize}
\item $v \in \mathbb{Z}^d$ is a row vector;
\item $w \in \mathbb{Z}^d$ is a column vector; and,
\item $\gamma:\{0,1\}^* \to \mathbb{Z}^{d \times d}$, is a homomorphism from the set of binary words
to the set of $d \times d$ integer matrices (that is, if $w = w_m w_{m-1}\cdots w_1$ is a binary
word, then $\gamma(w) = \gamma(w_m)\gamma(w_{m-1})\cdots\gamma(w_1)$).
\end{itemize}
Note that $\gamma$ is uniquely determined by the two matrices $\gamma(0)$ and $\gamma(1)$,
so from now on, we will instead write $(v,\gamma(0),\gamma(1),w)$ rather than $(v,\gamma,w)$.
The quadruple $(v,\gamma(0),\gamma(1),w)$ is a \emph{linear representation} for $a$ if, for all $n\geq 0$, we have
$a(n) = v\gamma([n]_2)w$, where $[n]_2$ is the binary representation of $n$.  The quantity $d$ is
the \emph{rank} of the linear representation.  If $a$ has a linear representation, then $a$ is
a \emph{$2$-regular sequence}.  The equivalence between this definition and the previous one
can be found in \cite[Theorem~16.1.3]{AS03}.

Any sum of the form \eqref{T(n)} defines a $2$-regular sequence and we can obtain a linear
representation for it using the computer package \texttt{Walnut}.  The key idea behind this comes from
the classical theorem of Lucas:
for integers $k$, $n$, and prime $p$, the following holds:
\begin{equation} \label{lem2eq1}
    \dbinom{n}{k}\equiv \prod_{i=1}^m\dbinom{n_i}{k_i}\pmod{p}
\end{equation}
where $[n]_p = n_m n_{m-1}\cdots n_1$ and $[k]_p = k_m k_{m-1} \cdots k_1$ are the base-$p$ expansions of
$n$ and $k$ respectively (if necessary, the shorter of the two base-$p$ expansions is padded with $0$'s on
the left so that both expansions have the same length).  Furthermore, we use the convention that
$\dbinom{n}{k}=0$ if $n<k$.

In this paper we will only consider $p=2$.  Note that in this case we have 
\begin{equation} \label{lem2eq2}
    \dbinom{n_i}{k_i} \equiv
    \begin{cases} 
        0 \pmod{2} & \text{if } n_i = 0,\; k_i = 1, \\
        1 \pmod{2} & otherwise.
    \end{cases}
\end{equation}
It follows that $\dbinom{n}{k}\pmod{2} \equiv 0\pmod{2}$ if and only if there exists $i$ such that
$[k_i,n_i]=[1,0]$.  This condition can be checked by the finite automaton given in Figure~\ref{binom2-aut}.
\begin{figure}[H]
\begin{center}
\includegraphics[width=3.5in]{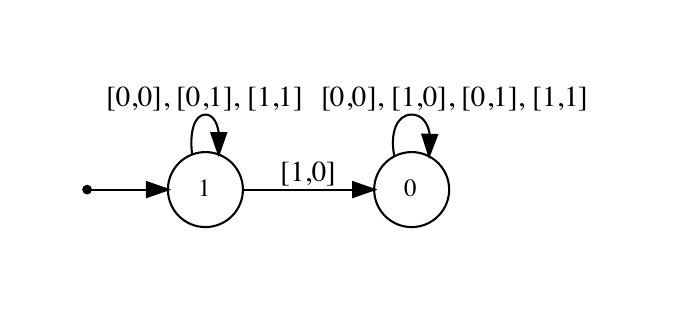}
\end{center}
\vskip -.4in
\caption{Automaton for $\dbinom{n}{k}$ modulo $2$}
\label{binom2-aut}
\end{figure}
This automaton reads pairs of digits $[k_i,n_i]$ and remains in state $1$ if no $[1,0]$ is seen;
otherwise, the automaton transitions to state $0$ and stays there once a $[1,0]$ is read.

Given such an automaton, the program \texttt{Walnut} can prove many things about the sequence computed
by the automaton, and, what is important for our purposes, it can compute linear representations
for sequences of the form \eqref{T(n)} (see the book by Shallit~\cite{Sha22} and in particular
Chapter~9 for details on how to use this program).  The Walnut command we use to compute
linear representations for \eqref{T(n)} is
\begin{verbatim}
eval [Sequence Name] n "?msd_2 (k <= n) &
    BINOM2[a_1*n+a_2*k][a_3*n+a_4*k]=@1 &
    BINOM2[n][k]=@1":
\end{verbatim}
(here \texttt{BINOM2} refers to the automaton given in Figure~\ref{binom2-aut}).  The output
of this command is a Maple program containing the linear representation of the sequence
\eqref{T(n)} (i.e., the triple $(v,\gamma,w)$ such that $T(n)=v\gamma([n]_2)w$).
(In fact, what the \texttt{eval} command does in this case is return a linear representation
for the sequence that counts, as a function of $n$, the number of $k$'s for which
the expression in quotation marks evaluates to \texttt{TRUE}.)

It is important to note that the linear representation that \texttt{Walnut} generates
may not have minimal rank.  However, there is an algorithm due to Schutzenberger and
presented in the book of Berstel and Reutenauer \cite[Section~2.3]{BR11},
that will take a linear representation of a regular sequence and produce a new
representation of minimal rank.  When we refer to ``minimizing'' a linear
representation, we mean applying this algorithm\footnote{Jeffrey Shallit has kindly
provided us with a Maple implementation of the minimization algorithm.}.

\section{The run length transform of a linear recurrence sequence}
In order to obtain Wu's results, we need to make the connection between the linear
representations for \eqref{T(n)} computed in the previous section and the run length
transform of linear recurrence sequences.

\begin{dfn} \label{Dfn1}
    Let $(S(n))_{n\geq0}$ be a sequence defined by:
    \begin{equation} \label{dfn1eq1}
        S(n+1) = d_0S(n) + \dots + d_rS(n-r)
    \end{equation}
    with
    \[ S(i) =
    \begin{cases} 
        1 & \text{if } i=0 \\
        c_i & \text{if } i = 1,..,r 
    \end{cases}
    \]
\end{dfn}

We define
\begin{align*}
    v &=
    \begin{bmatrix}
        1 & 0 & \dots & 0
    \end{bmatrix}_{1\times(r+1)}, \quad\quad
    w =
    \begin{bmatrix}
        1 \\ c_1 \\ \vdots \\ c_r
    \end{bmatrix}, \\
    \gamma(0) &=
    \begin{bmatrix}
            1 & 0 & \dots & 0 \\
            c_1 & 0 & \dots & 0 \\
            \vdots & \vdots & \dots & \vdots \\
            c_r & 0 & \dots & 0
        \end{bmatrix}, \quad
     \gamma(1) =
    \begin{bmatrix}
        0 \\ \vdots & I_{r\times r} \\ 0 \\ d_r & \dots & d_0
    \end{bmatrix}
\end{align*}
We will prove that $(v, \gamma(0), \gamma(1), w)$ is a linear representation
for the run length transform $T(n)$ of $S(n)$.

For the purposes of the following two results, let
\begin{align*}
    w_n :=
    \begin{bmatrix}
        S(n) \\ S(n+1) \\ \vdots \\ S(n+r)
    \end{bmatrix}, \; \text{where} \;
    w_0 =
    \begin{bmatrix}
        S(0) \\ S(1) \\ \vdots \\ S(r)
    \end{bmatrix} =
    \begin{bmatrix}
        1 \\ c_1 \\ \vdots \\ c_r
    \end{bmatrix} = w
\end{align*}

\begin{lem} \label{Lem1}
    $S(n) = v\gamma(1)^nw$.
\end{lem}
\begin{proof}
    Let $n \geq$ 1. Observe:
    \begin{align*}
        \gamma(1)w_n &=
        \begin{bmatrix}
            0 \\ \vdots & I_{r\times r} \\ 0 \\ d_r & \dots & d_0
        \end{bmatrix}
        \begin{bmatrix}
            S(n) \\ S(n+1) \\ \vdots \\ S(n+r)
        \end{bmatrix} \\
        &=
        \begin{bmatrix}
            S(n+1) \\ S(n+2) \\ \vdots \\ d_rS(n) + d_{r-1}S(n+1) + \dots + d_0S(n+r)
        \end{bmatrix} \\ &= w_{n+1}.
    \end{align*}
    Then
    \begin{equation} \label{lem1eq1}
        \gamma(1)^nw = \gamma(1)^{n-1}(\gamma(1)w) = \gamma(1)^{n-1}w_1 = \dots = w_n.
    \end{equation}
    Therefore,
    \begin{align*}
        v\gamma(1)^nw = vw_n =
        \begin{bmatrix}
            1 & 0 & \dots & 0
        \end{bmatrix}
        \begin{bmatrix}
            S(n) \\ S(n+1) \\ \vdots \\ S(n+r)
        \end{bmatrix} = S(n).
    \end{align*}
\end{proof}

\begin{thm} \label{Thm1}
    The run length transform $T(n)$ of the linear recurrence sequence $S(n)$ is
    a $2$-regular sequence.  In particular, we have $T(n) = v\gamma([n]_2)w$.
\end{thm}
\begin{proof}
    First, note that:
    \begin{equation} \label{thm1eq1}
        \gamma(0)^2 =
        \begin{bmatrix}
            1 & 0 & \dots & 0 \\
            c_1 & 0 & \dots & 0 \\
            \vdots & \vdots & \dots & \vdots \\
            c_r & 0 & \dots & 0
        \end{bmatrix}^2 = 
        \gamma(0)
    \end{equation}
    It follows that for any integer $k\geq 0$, $\gamma(0)^k = \gamma(0)$. Now let $n\geq0$, then:
    \begin{align*}
        \gamma(0)w_n &= 
         \begin{bmatrix}
            1 & 0 & \dots & 0 \\
            c_1 & 0 & \dots & 0 \\
            \vdots & \vdots & \dots & \vdots \\
            c_r & 0 & \dots & 0
        \end{bmatrix}
        \begin{bmatrix}
            S(n) \\ S(n+1) \\ \vdots \\ S(n+r)
        \end{bmatrix} \\ &= 
        \begin{bmatrix}
            S(n) \\ c_1S(n) \\ \vdots \\ c_rS(n)
        \end{bmatrix} \\ &= S(n)
        \begin{bmatrix}
            1 \\ c_1 \\ \vdots \\ c_r
        \end{bmatrix} \\ &= S(n)w.
    \end{align*}
    So
    \begin{equation} \label{thm1eq2}
        \gamma(0)w_n = S(n)w.
    \end{equation}
    Now let $[n]_2 = 1^{a_1}0^{b_1}\cdots1^{a_k}0^{b_k}$ for some $k\geq1$, and $a_i, b_i \geq 1$ for $i=1,\dots,k$, except possible $b_k = 0$. But if $b_k \neq 0$, then $\gamma(0)^{b_k}w = \gamma(0)w = S(0)w = w$, so without loss of generality, we may assume $b_k = 0$. Then using \eqref{lem1eq1}, \eqref{thm1eq1}, \eqref{thm1eq2}, and Lemma~\ref{Lem1}, we get our result:
    \begin{align*}
        v\gamma([n]_2)w &=v\gamma(1^{a_1}0^{b_1}\dotsm0^{b_{k-1}}1^{a_k})w \\
         &=v\gamma(1)^{a_1}\gamma(0)^{b_1}\dotsm\gamma(0)^{b_{k-1}}\gamma(1)^{a_k}w \\
         &=v\gamma(1)^{a_1}\gamma(0)\dotsm\gamma(0)\gamma(1)^{a_k}w \\
         &=v\gamma(1)^{a_1}\gamma(0)\dotsm\gamma(0)w_{a_k} \\
         &=v\gamma(1)^{a_1}\gamma(0)\dotsm\gamma(1)^{a_{k-1}}S(a_k)w \\
         & \vdots \\
         &=S(a_k)S(a_{k-1})\dotsm S(a_2)v\gamma(1)^{a_1}w \\
         &=S(a_k)S(a_{k-1})\dotsm S(a_2)S(a_1) \\ 
         &= \sum_{i \in \mathcal{L}(n)}S(i) = T(n)
    \end{align*}
\end{proof}
We can therefore generate the $n^{th}$ term of $S(n)$ by evaluating $v\gamma(1)^nw$ and the $n^{th}$ term of $T(n)$ by evaluating $v\gamma([n]_2)w$. So if a sequence $T(n)$ defined by \eqref{T(n)} has an associated quadruple $(v, \gamma(0), \gamma(1), w)$, then $T(n)$ is the run length transform of $S(n)$, where the coefficients of $S(n)$ are given by the bottom row of $\gamma(1)$, and the first $r+1$ terms are given by $w$.

\section{Wu's run length transforms}\label{sec2}
The theorems in this section are due to Wu~\cite{Wu22}, but are proved by obtaining a linear representation
using Walnut, minimizing the linear representation, if necessary, and observing that the resulting
linear representation gives the run length transform of the specified linear recurrence
sequence, as described in Theorem~\ref{Thm1}.  In this way we avoid the technical bitwise
arithmetic of Wu's proofs.

\begin{thm} \label{CWW thm6}
    Let $\displaystyle T(n) = \sum^n_{k=0} \Bigg[\dbinom{n-k}{2k}\dbinom{n}{k}\pmod{2}\Bigg]$. Then $T(n)$ is the
    run length transform of the Fibonacci sequence $1,1,2,3,5,8,\dots$ (OEIS~A000045).
\end{thm}
\begin{proof}
    Putting the above sequence into Walnut returns the quadruple $(v, \gamma(0), \gamma(1), w)$, with:
    \begin{align*} v =
        \begin{bmatrix}
            1 & 0
        \end{bmatrix}, \;
    w = 
        \begin{bmatrix}
            1 \\ 1
        \end{bmatrix}, \;
    \gamma(0) = 
        \begin{bmatrix}
            1 & 0 \\ 1 & 0
        \end{bmatrix}, \;
    \gamma(1) = 
        \begin{bmatrix}
            0 & 1 \\ 1 & 1
        \end{bmatrix}
    \end{align*}
    Thus by Lemma~\ref{Lem1} and Theorem~\ref{Thm1}, $T(n)$ is the run length transform of the sequence defined by $S(n) = S(n-1) + S(n-2)$ for $n\geq2$, with $S(0) = S(1) = 1$. This is precisely the Fibonacci sequence, which confirms the result.
\end{proof}
\begin{thm} \label{CWW thm7}
    Let $\displaystyle T(n) = \sum^n_{k=0} \Bigg[\dbinom{3k}{k}\dbinom{n}{k}\pmod{2}\Bigg]$. Then $T(n)$ is the run length transform of the truncated Fibonacci sequence $1,2,3,5,8,13,\dots$
\end{thm}
\begin{proof}
     Putting the above sequence into Walnut, and then minimizing it, we get the quadruple $(v, \gamma(0), \gamma(1), w)$, with:
    \begin{align*} v =
        \begin{bmatrix}
            1 & 0
        \end{bmatrix}, \;
    w = 
        \begin{bmatrix}
            1 \\ 2
        \end{bmatrix}, \;
    \gamma(0) = 
        \begin{bmatrix}
            1 & 0 \\ 2 & 0
        \end{bmatrix}, \;
    \gamma(1) = 
        \begin{bmatrix}
            0 & 1 \\ 1 & 1
        \end{bmatrix}
    \end{align*}
    Thus by Lemma~\ref{Lem1} and Theorem~\ref{Thm1}, $T(n)$ is the run length transform of the sequence defined by $S(n) = S(n-1) + S(n-2)$ for $n\geq2$, with $S(0) = 1$, $S(1) = 2$. This is precisely the truncated Fibonacci sequence, which confirms the result.
\end{proof}
\begin{thm} \label{CWW thm8}
    Let $\displaystyle T(n)$ = $\sum^n_{k=0}\Bigg[\dbinom{n}{2k}\dbinom{n}{k}\pmod{2}\Bigg]$. Then $T(n)$ is the run length transform of the sequence $1,1,2,4,8,16,\dots,$ i.e., $1$ followed by the positive powers of $2$ (OEIS~A000079).
\end{thm}
\begin{proof}
     Putting the above sequence into Walnut returns the quadruple $(v, \gamma(0), \gamma(1), w)$, with:
    \begin{align*} v =
        \begin{bmatrix}
            1 & 0
        \end{bmatrix}, \;
    w = 
        \begin{bmatrix}
            1 \\ 1
        \end{bmatrix}, \;
    \gamma(0) = 
        \begin{bmatrix}
            1 & 0 \\ 1 & 0
        \end{bmatrix}, \;
    \gamma(1) = 
        \begin{bmatrix}
            0 & 1 \\ 0 & 2
        \end{bmatrix}
    \end{align*}
    Thus by Lemma~\ref{Lem1} and Theorem~\ref{Thm1}, $T(n)$ is the run length transform of the sequence defined by $S(n) = 2S(n-1)$ for $n\geq2$, with $S(0) = S(1) = 1$. This is precisely the sequence $1$ followed by the positive powers of $2$, which confirms the result.
\end{proof}
\begin{thm} \label{CWW thm9}
    Let $\displaystyle T(n) = \sum^n_{k=0} \Bigg[\dbinom{n+2k}{2k}\dbinom{n}{k}\pmod{2}\Bigg]$. Then $T(n)$ is the run length transform of the sequence $1,2,2,2,2,2,\dots,$ i.e., the sequence of $2$'s prepended with a $1$ (OEIS~A040000).
\end{thm}
\begin{proof}
     Putting the above sequence into Walnut, and then minimizing it, we get the quadruple $(v, \gamma(0), \gamma(1), w)$, with:
    \begin{align*} v =
        \begin{bmatrix}
            1 & 0
        \end{bmatrix}, \;
    w = 
        \begin{bmatrix}
            1 \\ 2
        \end{bmatrix}, \;
    \gamma(0) = 
        \begin{bmatrix}
            1 & 0 \\ 2 & 0
        \end{bmatrix}, \;
    \gamma(1) = 
        \begin{bmatrix}
            0 & 1 \\ 0 & 1
        \end{bmatrix}
    \end{align*}
    Thus by Lemma~\ref{Lem1} and Theorem~\ref{Thm1}, $T(n)$ is the run length transform of the sequence defined by $S(n) = S(n-1)$ for $n\geq2$, with $S(0) = 1$, $S(1) = 2$. This is precisely the sequence of $2$'s, prepended with a $1$, which confirms the result.
\end{proof}
\begin{thm} \label{CWW thm10}
    Let $\displaystyle T(n) = \sum^n_{k=0} \Bigg[\dbinom{n+k}{n-k}\dbinom{n}{k}\pmod{2}\Bigg]$. Then $T(n)$ is the run length transform of the sequence of positive integers $1,2,3,4,5,6,\dots$ (OEIS~A000027).
\end{thm}
\begin{proof}
     Putting the above sequence into Walnut, and then minimizing it, we get the quadruple $(v, \gamma(0), \gamma(1), w)$, with:
    \begin{align*} v =
        \begin{bmatrix}
            1 & 0
        \end{bmatrix}, \;
    w = 
        \begin{bmatrix}
            1 \\ 2
        \end{bmatrix}, \;
    \gamma(0) = 
        \begin{bmatrix}
            1 & 0 \\ 2 & 0
        \end{bmatrix}, \;
    \gamma(1) = 
        \begin{bmatrix}
            0 & 1 \\ -1 & 2
        \end{bmatrix}
    \end{align*}
    Thus by Lemma~\ref{Lem1} and Theorem~\ref{Thm1}, $T(n)$ is the run length transform of the sequence defined by $S(n) = 2S(n-1) - S(n-2)$ for $n\geq2$, with $S(0) = 1$, $S(1) = 2$. This is precisely the sequence of positive integers, which confirms the result.
\end{proof}
\begin{thm} \label{CWW thm14}
    Let $\displaystyle T(n) = \sum^n_{k=0}\Bigg[\dbinom{n-k}{6k}\dbinom{n}{k}\pmod{2}\Bigg]$. Then $T(n)$ is the run length transform of Narayana's cows sequence $1,1,1,2,3,4,6,9,\dots$ (OEIS~A000930).
\end{thm}
\begin{proof}
     Putting the above sequence into Walnut returns the quadruple $(v, \gamma(0), \gamma(1), w)$, with:
    \begin{align*} v =
        \begin{bmatrix}
            1 & 0 & 0
        \end{bmatrix}, \;
    w = 
        \begin{bmatrix}
            1 \\ 1 \\ 1
        \end{bmatrix}, \;
    \gamma(0) = 
        \begin{bmatrix}
            1 & 0 & 0\\ 1 & 0 & 0 \\ 1 & 0 & 0
        \end{bmatrix}, \;
    \gamma(1) = 
        \begin{bmatrix}
            0 & 1 & 0 \\ 0 & 0 & 1 \\ 1 & 0 & 1
        \end{bmatrix}
    \end{align*}
    Thus by Lemma~\ref{Lem1} and Theorem~\ref{Thm1}, $T(n)$ is the run length transform of the sequence defined by $S(n) = S(n-1) + S(n-3)$ for $n\geq3$, with $S(0) = S(1) = S(2) = 1$. This is precisely Narayana's cows sequence, which confirms the result.
\end{proof}
\begin{thm} \label{CWW thm15}
    Let $\displaystyle T(n) = \sum^n_{k=0}\Bigg[\dbinom{n+3k}{6k}\dbinom{n}{k}\pmod{2}\Bigg]$. Then $T(n)$ is the run length transform of the sequence of doubled positive integers, i.e., $1,1,2,2,3,3,4,4,\dots$ (OEIS~A008619).
\end{thm}
\begin{proof}
     Putting the above sequence into Walnut  and then minimizing it, we get the quadruple $(v, \gamma(0), \gamma(1), w)$, with:
    \begin{align*} v =
        \begin{bmatrix}
            1 & 0 & 0
        \end{bmatrix}, \;
    w = 
        \begin{bmatrix}
            1 \\ 1 \\ 2
        \end{bmatrix}, \;
    \gamma(0) = 
        \begin{bmatrix}
            1 & 0 & 0\\ 1 & 0 & 0 \\ 2 & 0 & 0
        \end{bmatrix}, \;
    \gamma(1) = 
        \begin{bmatrix}
            0 & 1 & 0 \\ 0 & 0 & 1 \\ -1 & 1 & 1
        \end{bmatrix}
    \end{align*}
    Thus by Lemma~\ref{Lem1} and Theorem~\ref{Thm1}, $T(n)$ is the run length transform of the sequence defined by $S(n) = -S(n-1) +S(n-2) + S(n-3)$ for $n\geq3$, with $S(0) = S(1) = 1$, $S(2) = 2$. This is precisely the sequence of doubled positive integers, which confirms the result.
\end{proof}
\begin{thm} \label{CWW thm17}
    Let $\displaystyle T(n) = \sum^n_{k=0}\Bigg[\dbinom{n+2k}{2n-k}\dbinom{n}{k}\pmod{2}\Bigg]$. Then $T(n)$ is the run length transform of the Lucas numbers, prepended with the terms $1,1$, i.e., the sequence $1,1,2,1,3, 4,7,11,\dots$ (OEIS~A329723).
\end{thm}
\begin{proof}
     Putting the above sequence into Walnut  and then minimizing it, we get the quadruple $(v, \gamma(0), \gamma(1), w)$, with:
    \begin{align*} v =
        \begin{bmatrix}
            1 & 0 & 0 & 0
        \end{bmatrix}, \;
    w = 
        \begin{bmatrix}
            1 \\ 1 \\ 2 \\ 1
        \end{bmatrix}, \;
    \gamma(0) = 
        \begin{bmatrix}
            1 & 0 & 0 & 0 \\ 1 & 0 & 0 & 0 \\ 2 & 0 & 0 & 0 \\ 1 & 0 & 0 & 0
        \end{bmatrix}, \;
    \gamma(1) = 
        \begin{bmatrix}
            0 & 1 & 0 & 0 \\ 0 & 0 & 1 & 0 \\ 0 & 0 & 0 & 1 \\ 0 & 0 & 1 & 1
        \end{bmatrix}
    \end{align*}
    Thus by Lemma~\ref{Lem1} and Theorem~\ref{Thm1}, $T(n)$ is the run length transform of the sequence defined by $S(n) = S(n-1) + S(n-2)$ for $n\geq4$, with $S(0) = S(1) = 1$, $S(2) = 2$, $S(3) = 1$. This is precisely the Lucas numbers prepended by the terms $1,1$.
\end{proof}

\section{New run length transforms} \label{sec3}
In this section we give some new run length transforms. 
\begin{thm} \label{new RLT1}
    Let $\displaystyle T(n) = \sum^n_{k=0}\left[\dbinom{n+5k}{2n+2k}\dbinom{n}{k}\pmod{2}\right]$. Then $T(n)$ is the run length transform of the sequence generated by $S(n) = S(n-1) + S(n-2) - S(n-3) + S(n-4)$ for $n\geq4$, with $S(0) = S(1) = S(2) = S(3) = 1$; i.e., the sequence $1,1,1,1,2,3,5,7,11,16,25,\dots$.
\end{thm}
\begin{proof}
     Putting the above sequence into Walnut  and then minimizing it, we get the quadruple $(v, \gamma(0), \gamma(1), w)$, with:
    \begin{align*} v =
        \begin{bmatrix}
            1 & 0 & 0 & 0
        \end{bmatrix}, \;
    w = 
        \begin{bmatrix}
            1 \\ 1 \\ 1 \\ 1
        \end{bmatrix}, \;
    \gamma(0) = 
        \begin{bmatrix}
            1 & 0 & 0 & 0 \\ 1 & 0 & 0 & 0 \\ 1 & 0 & 0 & 0 \\ 1 & 0 & 0 & 0
        \end{bmatrix}, \;
    \gamma(1) = 
        \begin{bmatrix}
            0 & 1 & 0 & 0 \\ 0 & 0 & 1 & 0 \\ 0 & 0 & 0 & 1 \\ 1 & -1 & 1 & 1
        \end{bmatrix}
    \end{align*}
    Thus by Lemma~\ref{Lem1} and Theorem~\ref{Thm1}, $T(n)$ is the run length transform of the sequence defined by $S(n) = S(n-1) + S(n-2) - S(n-3) + S(n-4)$ for $n\geq4$, with $S(0) = S(1) = S(2) = S(3) = 1$; i.e., the sequence $1,1,1,1,2,3,5,7,11,16,25,\dots$.
\end{proof}
\begin{thm} \label{new RLT2}
    Let $\displaystyle T(n) = \sum^n_{k=0}\Bigg[\dbinom{n+5k}{2k}\dbinom{n}{k}\pmod{2}\Bigg]$. Then $T(n)$ is the run length transform of the sequence generated by $S(n) = S(n-1) + S(n-2) - S(n-3) + S(n-4)$, with $S(0) = 1$, $S(1) = S(2) = 2$, $S(3) = 3$.
\end{thm}
\begin{proof}
     Putting the above sequence into Walnut  and then minimizing it, we get the quadruple $(v, \gamma(0), \gamma(1), w)$, with:
    \begin{align*} v =
        \begin{bmatrix}
            1 & 0 & 0 & 0
        \end{bmatrix}, \;
    w = 
        \begin{bmatrix}
            1 \\ 2 \\ 2 \\ 3
        \end{bmatrix}, \;
    \gamma(0) = 
        \begin{bmatrix}
            1 & 0 & 0 & 0 \\ 2 & 0 & 0 & 0 \\ 2 & 0 & 0 & 0 \\ 3 & 0 & 0 & 0
        \end{bmatrix}, \;
    \gamma(1) = 
        \begin{bmatrix}
            0 & 1 & 0 & 0 \\ 0 & 0 & 1 & 0 \\ 0 & 0 & 0 & 1 \\ 1 & -1 & 1 & 1
        \end{bmatrix}
    \end{align*}
    Thus by Lemma~\ref{Lem1} and Theorem~\ref{Thm1}, $T(n)$ is the run length transform of the sequence defined by $S(n) = S(n-1) + S(n-2) - S(n-3) + S(n-4)$ for $n\geq4$, with $S(0) = 1$, $S(1) = S(2) = 2$, $S(3) = 3$.
\end{proof}
Note that the previous two sequences are generated by the same rule, and differ only by their starting terms. 
\begin{thm} \label{new RLT3}
    Let $\displaystyle T(n) = \sum^n_{k=0}\Bigg[\dbinom{-n+7k}{n+k}\dbinom{n}{k}\pmod{2}\Bigg]$.
    Then $T(n)$ is the run length transform of the Padovan numbers (OEIS~A000931) starting with offset $5$, i.e. the sequence $1,1,1,2,2,3,4,5,\dots$
\end{thm}
\begin{proof}
     Putting the above sequence into Walnut  and then minimizing it, we get the quadruple $(v, \gamma(0), \gamma(1), w)$, with:
    \begin{align*} v =
        \begin{bmatrix}
            1 & 0 & 0
        \end{bmatrix}, \;
    w = 
        \begin{bmatrix}
            1 \\ 1 \\ 1
        \end{bmatrix}, \;
    \gamma(0) = 
        \begin{bmatrix}
            1 & 0 & 0\\ 1 & 0 & 0 \\ 1 & 0 & 0
        \end{bmatrix}, \;
    \gamma(1) = 
        \begin{bmatrix}
            0 & 1 & 0 \\ 0 & 0 & 1 \\ 1 & 1 & 0
        \end{bmatrix}
    \end{align*}
    Thus by Lemma~\ref{Lem1} and Theorem~\ref{Thm1}, $T(n)$ is the run length transform of the sequence defined by $S(n) = S(n-2) + S(n-3)$ for $n\geq3$, with $S(0) = S(1) = S(2) = 1$. This is precisely the Padovan numbers starting with offset $5$.
\end{proof}
\begin{thm} \label{new RLT4}
    Let $\displaystyle T(n) = \sum^n_{k=0}\Bigg[\dbinom{n+7k}{3n+k}\dbinom{n}{k}\pmod{2}\Bigg]$.
    Then $T(n)$ is the run length transform of the Padovan numbers (OEIS~A000931) starting with offset $3$, i.e. the sequence $1,0,1,1,1,2,2,3,\dots$
\end{thm}
\begin{proof}
     Putting the above sequence into Walnut  and then minimizing it, we get the quadruple $(v, \gamma(0), \gamma(1), w)$, with:
    \begin{align*} v =
        \begin{bmatrix}
            1 & 0 & 0
        \end{bmatrix}, \;
    w = 
        \begin{bmatrix}
            1 \\ 0 \\ 1
        \end{bmatrix}, \;
    \gamma(0) = 
        \begin{bmatrix}
            1 & 0 & 0\\ 1 & 0 & 0 \\ 1 & 0 & 0
        \end{bmatrix}, \;
    \gamma(1) = 
        \begin{bmatrix}
            0 & 1 & 0 \\ 0 & 0 & 1 \\ 1 & 1 & 0
        \end{bmatrix}
    \end{align*}
    Thus by Lemma~\ref{Lem1} and Theorem~\ref{Thm1}, $T(n)$ is the run length transform of the sequence defined by $S(n) = S(n-2) + S(n-3)$ for $n\geq3$, with $S(0) =1$, $S(1) = 0$, $S(2) = 1$. This is precisely the Padovan numbers starting with offset $3$.
\end{proof}
\begin{thm} \label{new RLT5}
    Let $\displaystyle T(n) = \sum^n_{k=0}\Bigg[\dbinom{6k}{n+3k}\dbinom{n}{k}\pmod{2}\Bigg]$. Then $T(n)$ is the run length transform of the sequence given by alternating between $1$ and the natural numbers, i.e. the sequence $1,1,1,2,1,3,1,4,\dots$
\end{thm}
\begin{proof}
     Putting the above sequence into Walnut  and then minimizing it, we get the quadruple $(v, \gamma(0), \gamma(1), w)$, with:
    \begin{align*} v =
        \begin{bmatrix}
            1 & 0 & 0 & 0
        \end{bmatrix}, \;
    w = 
        \begin{bmatrix}
            1 \\ 1 \\ 1 \\ 2
        \end{bmatrix}, \;
    \gamma(0) = 
        \begin{bmatrix}
            1 & 0 & 0 & 0 \\ 1 & 0 & 0 & 0 \\ 1 & 0 & 0 & 0 \\ 2 & 0 & 0 & 0
        \end{bmatrix}, \;
    \gamma(1) = 
        \begin{bmatrix}
            0 & 1 & 0 & 0 \\ 0 & 0 & 1 & 0 \\ 0 & 0 & 0 & 1 \\ -1 & 0 & 2 & 0
        \end{bmatrix}
    \end{align*}
    Thus by Lemma~\ref{Lem1} and Theorem~\ref{Thm1}, $T(n)$ is the run length transform of the sequence defined by $S(n) = 2S(n-2) - S(n-4)$ for $n\geq4$, with $S(0) = S(1) = S(2) = 1$, $S(3) = 2$, which is the sequence alternating between $1$ and the natural numbers.
\end{proof}
\begin{thm} \label{new RLT6}
    Let $\displaystyle T(n) = \sum^n_{k=0}\Bigg[\dbinom{-2n+8k}{n+k}\dbinom{n}{k}\pmod{2}\Bigg]$.
    Then $T(n)$ is the run length transform of the sequence with period $1,1,0$. 
\end{thm}
\begin{proof}
     Putting the above sequence into Walnut  and then minimizing it, we get the quadruple $(v, \gamma(0), \gamma(1), w)$, with:
    \begin{align*} v =
        \begin{bmatrix}
            1 & 0 & 0
        \end{bmatrix}, \;
    w = 
        \begin{bmatrix}
            1 \\ 1 \\ 0
        \end{bmatrix}, \;
    \gamma(0) = 
        \begin{bmatrix}
            1 & 0 & 0\\ 1 & 0 & 0 \\ 0 & 0 & 0
        \end{bmatrix}, \;
    \gamma(1) = 
        \begin{bmatrix}
            0 & 1 & 0 \\ 0 & 0 & 1 \\ 1 & 0 & 0
        \end{bmatrix}
    \end{align*}
    Thus by Lemma~\ref{Lem1} and Theorem~\ref{Thm1}, $T(n)$ is the run length transform of the sequence defined by $S(n) = S(n-3)$ for $n\geq3$, with $S(0) = S(1) = 1$, $S(2) = 0$, which is the sequence with period $1,1,0$.
\end{proof}

\section{Average value of select run length transforms} \label{sec4}
One of the advantages of using linear representations to describe the run length transforms $T(n)$ that we have given
in the previous sections is that we can easily compute the average value of $T(n)$ in the interval $2^r \leq n < 2^{r+1}$
for $r \geq 0$.  The method is described, with many examples, in Shallit~\cite[Sections~9.8--9.11]{Sha22}.

We will give the details of this calculation for the sequence given in Theorem~\ref{CWW thm6}, and simply state the rest.
Let $(v, \gamma(0), \gamma(1), w)$ be as in Theorem~\ref{CWW thm6}, and define
\begin{equation*}
    M = \gamma(0) + \gamma(1) = 
    \begin{bmatrix}
        1 & 0 \\ 1 & 0
    \end{bmatrix}
    +
    \begin{bmatrix}
        0 & 1 \\ 1 & 1
    \end{bmatrix}
    =
    \begin{bmatrix}
        1 & 1 \\ 2 & 1
    \end{bmatrix}.
\end{equation*}
Let $g(r)$ denote the sum of the first $2^r$ terms of $T(n)$.  Then
\begin{align*}
    g(r) &= \sum_{0\leq n<2^r}T(n)\\
    &= \sum_{0\leq n<2^r}v\gamma([n]_2)w\\
    &= v\left(\sum_{0\leq n<2^r}\gamma([n]_2)\right)w\\
    &= v\left(\sum_{z \in \{0,1\}^n}\gamma(z)\right)w\\
    &= v\left(\gamma(0)+\gamma(1)\right)^n w\\
    &= vM^n w.
\end{align*}    
The minimal polynomial for $M$ is
\begin{align*}
    p(x) = x^2 -2x -1 = (x-(1+\sqrt{2}))(x-(1-\sqrt{2})).
\end{align*}
It follows that $g(r)$ is given by the exponential polynomial
\begin{align*}
    g(r) = c_1(1+\sqrt{2})^r + c_2(1-\sqrt{2})^r,
\end{align*}
for some coefficients $c_1, c_2$, which we determine by solving the linear system
\begin{align*}
    g(0) &= 1 = c_1 + c_2\\
    g(1) &= 2 = c_1(1+\sqrt{2}) + c_2(1-\sqrt{2}).
\end{align*}
We get
\begin{align*}
    c_1 = \frac{1+\sqrt{2}}{2\sqrt{2}}, \quad c_2 = -\frac{1-\sqrt{2}}{2\sqrt{2}},
\end{align*}
and hence,
\begin{align*}
    g(r) &=  \frac{1+\sqrt{2}}{2\sqrt{2}}(1+\sqrt{2})^r - \frac{1-\sqrt{2}}{2\sqrt{2}}(1-\sqrt{2})^r \\
    &= \frac{1}{2\sqrt{2}}\left((1+\sqrt{2})^{r+1} - (1-\sqrt{2})^{r+1}\right).
\end{align*}
Next, we want to find $g(r+1) - g(r)$, which is the sum of the terms of $T(n)$ in the range $2^r\leq n < 2^{r+1}$:
\begin{align*}
    g(r+1) - g(r) &= \frac{1}{2\sqrt{2}}\left(\left((1+\sqrt{2})^{r+2} - (1-\sqrt{2})^{r+2}\right) - \left((1+\sqrt{2})^{r+1} - (1-\sqrt{2})^{r+1}\right)\right) \\ 
    &= \left(\left((1+\sqrt{2})^{r+2} - (1+\sqrt{2})^{r+1}\right) - \left((1-\sqrt{2})^{r+2} - (1-\sqrt{2})^{r+1}\right)\right) \\
    &= \frac{1}{2\sqrt{2}}\left((1+\sqrt{2}-1)(1+\sqrt{2})^{r+1} - (1-\sqrt{2}-1)(1-\sqrt{2})^{r+1}\right) \\
    &= \frac{1}{2}\left((1+\sqrt{2})^{r+1} + (1-\sqrt{2})^{r+1}\right).
\end{align*}
Finally, the average value of $T(n)$ in the range $2^r\leq n<2^{r+1}$ is given by
\begin{align*}
    \mu(r) = \frac{g(r+1)-g(r)}{2^r} = \frac{1}{2^{r+1}}\bigg((1+\sqrt{2})^{r+1} + (1-\sqrt{2})^{r+1}\bigg).
\end{align*}

We now summarize the average values for some of Wu's run length transforms (see Section~\ref{sec2}) below,
where the coefficients $a_i$ refer to the parameters in \eqref{T(n)}.
We omit the average value results for the results involving recurrence relations of order $3$ or higher,
since the roots of the associated minimal polynomials can't be displayed as nicely, but the reader
can easily carry out the above method for those sequences as well.
\begin{center}
\begin{tabular}{ |c|c|c| } 
\hline
Reference & $(a_1, a_2, a_3, a_4)$ & Average Value $\mu(r)$ \\
\hline
Thm.~\ref{CWW thm6} & $(1,-1,0,2)$ & $\frac{1}{2^{r+1}}\big((1+\sqrt{2})^{r+1} + (1-\sqrt{2})^{r+1}\big)$\\ 
Thm.~\ref{CWW thm7} & $(0,3,0,1)$ & $\frac{1}{2^{r+1}}\big((2+\sqrt{3})(1+\sqrt{3})^r + (2-\sqrt{3})(1-\sqrt{3})^r\big)$  \\ 
Thm.~\ref{CWW thm8} & $(1,0,0,2)$ & $\frac{1}{2^{2r+2}\sqrt{5}}\big((1+\sqrt{5})^2(3+\sqrt{5})^r-(1-\sqrt{5})^2(3-\sqrt{5})^r\big)$ \\ 
Thm.~\ref{CWW thm9} & $(1,2,0,2)$ & $\frac{\sqrt{2}}{2^{r+1}}\big((1+\sqrt{2})^{r+1} - (1-\sqrt{2})^{r+1}\big)$ \\
Thm.~\ref{CWW thm10} & $(1,1,1,-1)$ & $\frac{1}{2^{2r+2}\sqrt{5}}\big((1+\sqrt{5})(3+\sqrt{5})^{r+1}-(1-\sqrt{5})(3-\sqrt{5})^{r+1}\big)$ \\
\hline
\end{tabular}
\end{center}

\section{Characterizing a family of run length transforms} \label{sec5}
In this section we analyze the family of sequences defined by
$$T_m(n) = \sum^n_{k=0} \left[\binom{2^mk}{n+k}\binom{n}{k}\pmod{2}\right],$$
for $m\geq2$. We claim that $T_m(n)$ is the run length transform of the sequence defined by
$S_m(n) = S_m(n-m)$ for $n\geq m$, with $S_m(0) = 1, S_m(1) = S_m(2) = \dots = S_m(m-1) = 0$,
or, in other words, the sequence with the period of $1$ followed by $m-1$ $0$'s. 

In the proofs that follow, we will be looking for integers $k\leq n$ such that $\binom{2^mk}{n+k}\equiv1\pmod{2}$. If for a particular $k$, there exists $i$ such that the $i^{th}$ bit of $[2^mk]_2$ is $0$, but the $i^{th}$ bit of $[n+k]_2$ is $1$, then $\binom{2^mk}{n+k} \equiv 0\pmod{2}$ by \eqref{lem2eq2}, and so we say the $i^{th}$ bit $\textbf{fails}$.  Here we count bits from right to left, starting with index $1$.

Consider the case where $n = 2^\ell - 1$ for some $\ell \geq 0$. Then $[n]_2 = 1^\ell$, so $\mathcal{L}(n) = \{\ell\}$. Thus $T_m(n) = S_m(\ell)$ in such cases. 
\begin{lem} \label{Lem2}
    Let $n = 2^\ell - 1$ for some $\ell \geq 0$.
    Then \[T_m(n) =
    \begin{cases}
        1 & \text{if } \ell \equiv 0\pmod{m}, \\
        0 & \text{otherwise.}
    \end{cases}\]
\end{lem}
\begin{proof}
    We will show that for each $n$ of the form $n = 2^\ell - 1$ with $\ell \equiv 0\pmod{m}$, there is exactly one $k$ such that $\binom{2^mk}{n+k} \equiv 1\pmod{2}$, namely, $k = \frac{2^\ell-1}{2^m-1}$, and that only $n$'s of this form have such a $k$. In fact, $\frac{2^\ell-1}{2^m-1}$ is an integer only when $\ell \equiv 0\pmod{m}$. So such a $k$ only exists under these conditions, and it is this observation which motivates the proof.
    
    First, note that $[n]_2 = 1^\ell$ implies $\binom{n}{k} \equiv 1\pmod{2}$ for $k \leq n$ by \eqref{lem2eq1} and \eqref{lem2eq2}. So $T_m(n) =  \sum^n_{k=0} \big[\binom{2^mk}{n+k}\pmod{2}\big]$. Consider some $k \leq n$, where the length of $[k]_2$ is $r$. If $2^mk < n+k$, then $\binom{2^mk}{n+k} = 0$, so assume $2^mk \geq n+k$, or equivalently, 
    \begin{equation} \label{lem2eq3}
        \frac{n}{2^m-1} \leq k \leq n.
    \end{equation}
    Moreover, this also implies $r \leq \ell$.
    
    Clearly, if $\ell = 0$, then $n = 0$, and $T_m(0) = 1$.
    Suppose that $0 < \ell < m$.  Then by \eqref{lem2eq3}, we have $k > 0$. But the first $m$ bits of $[2^mk]_2$ will all be $0$'s, and $[n+k]_2$ will have length $\ell + 1$, where the $(\ell + 1)^{th}$ bit is $1$. However, we have $\ell+1 \leq m$, which means that the $(\ell + 1)^{th}$ bit of $[2^mk]_2$ is a $0$, and hence the $(\ell + 1)^{th}$ bit fails. So for every $n = 2^\ell - 1$ with $0 < \ell < m$, there is no $k \leq n$ such that $\binom{2^mk}{n+k} \not\equiv 0\pmod{2}$. So the claim holds for $\ell < m$.

    Write
    \begin{align*}
        [k]_2 &= a_\ell\cdots a_1,\\
        [n+k]_2 &= b_{\ell+1}\cdots b_1,\text{ and}\\
        [2^mk]_2 &= c_{\ell+m}\cdots c_1,
    \end{align*}
    where $a_\ell, \ldots, a_{\ell-m+1}$ may be $0$'s if necessary.  Of course, we have $c_{\ell+m}\cdots c_{m+1} = a_\ell\cdots a_1$ and
    $c_m\cdots c_1 = 0^m$.  Furthermore, we suppose there is no index $i$ such that $c_i=0$ and $b_i=1$.

    Note that since $[n]_2=1^\ell$, when adding $n+k$ there is a carry at every position after the first.
    Since $c_m\cdots c_1 = 0^m$, we have $b_m\cdots b_1 = 0^m$.  However, this forces $a_m\cdots a_1 = 0^{m-1}1$.
    Now since $a_m\cdots a_2 = c_{2m}\cdots c_{m+2} = 0^{m-1}$, we have $b_{2m}\cdots b_{m+2} = 0^{m-1}$.
    However, this forces $a_{2m}\cdots a_{m+2} = 0^{m-1}$.  Continuing in this manner, we find that
    $$b_{mi}\cdots b_{m(i-1)+2} = a_{mi}\cdots a_{m(i-1)+2} = 0^{m-1}$$
    for $i=1,\ldots,\lfloor \ell/m \rfloor.$

   We also have $b_{\ell+1}=1$, which implies $\ell \equiv 0 \pmod{m}$, since otherwise we would have
   $\ell+1-m\not\equiv 1\pmod{m}$ and hence $a_{\ell+1-m}=c_{\ell+1}=0$, which is a contradiction.
   Indeed we must have $a_{\ell+1-m}=c_{\ell+1}=1$, which forces $b_{\ell+1-m}=1$.  Continuing in this way
   we find that $b_{\ell+1-2m}=a_{\ell+1-2m}=1$, and so on; i.e., that
   $$a_{m(i-1)+1}=b_{m(i-1)+1}=1$$
   for $i=1,\ldots,\ell/m$.

   We have thus seen that $\binom{2^mk}{n+k} \equiv 1\pmod{2}$ exactly when $\ell \equiv 0 \pmod{m}$ and
   $[k]_2 = (0^{m-1}1)^{\ell/m}$; i.e., that
   \begin{align*}
       k = \sum_{i=0}^{\frac{\ell}{m}-1}2^{mi} = \frac{2^\ell-1}{2^m-1}.
   \end{align*}

\end{proof}
Thus we have that $T_m(2^\ell-1) = S_m(\ell) = 1$. In order to prove $T_m(n)$ is indeed the run length transform of $S_m(n)$, it is sufficient to prove the following:
\begin{thm} \label{Thm2}
Let $n\geq0$. Then
    \begin{equation*} 
    T_m(n) = 
        \begin{cases}
            1 & \text{if each run of 1's in } [n]_2 \text{ has length divisible by $m$}, \\
            0 & \text{otherwise}.
        \end{cases}
    \end{equation*}
\end{thm}
\begin{proof}
    Let $n\geq0$, where the length of $[n]_2 = \ell$. We've already covered the case where $[n]_2$ is a run of $\ell$ $1$'s, so assume the $j^{th}$ bit of $[n]_2$ is $0$. Then any $k$ with a $1$ in the $j^{th}$ bit will fail, as by \eqref{lem2eq1}:
    \begin{align*}
        &\dbinom{n_j}{k_j} = \dbinom{0}{1} = 0 \\ \implies &\dbinom{n}{k} \equiv 0 \pmod{2} \\
        \implies &\dbinom{2^mk}{n+k}\dbinom{n}{k}\equiv 0 \pmod{2}
    \end{align*}
    So we only need to examine $k \leq n$ such that
    \begin{equation*} 
        \forall i \leq \ell, \;\;\; n_i = 0 \implies k_i = 0
    \end{equation*}

    If $[n]_2$ only has a single run of $1$'s, say $[n]_2 = 1^{\ell_1}0^{t_1}$, where $t_1 \geq 0$,
    then by the proof of Lemma~\ref{Lem2}, we have $\ell_1 \equiv 0 \pmod{m}$ and the only possible
    $k$ that contributes a non-zero value to the sum defining $T_m(n)$ is $[k]_2 = (0^{m-1}1)^{\ell_1/m}0^{t_1}$.
    So suppose $[n]_2$ has two runs of $1$'s, say
    $$[n]_2 = 1^{\ell_2}0^{t_2}1^{\ell_1}0^{t_1},\quad t_2 \geq 1, t_1 \geq 0.$$
    Again, by Lemma~\ref{Lem2} and its proof we have $\ell_1 \equiv 0 \pmod{m}$ and the only possible
    $k$ that contributes a non-zero value to the sum defining $T_m(n)$ has the form
    $$[k]_2 = a_r\cdots a_{r-\ell_2+1}0^{t_2}(0^{m-1}1)^{\ell_1/m}0^{t_1},$$
    where $r=\ell_2+t_2+\ell_1+t_1$.  Note that when adding $n+k$, there is a carry at position
    $\ell_1+t_1+1$ but there is no carry at position $r-\ell_2+1$.  Furthermore,
    since $t_2 \geq 1$, there are at least $m$ $0$'s to the right of $a_r\cdots a_{r-\ell_2+1}$ in $[k]_2$.
    We can therefore apply the same analysis from the proof of Lemma~\ref{Lem2} to the run
    $1^{\ell_2}$ in $[n]_2$ and we find that $\ell_2 \equiv 0 \pmod{m}$ and
    $$[k]_2 = (0^{m-1}1)^{\ell_2/m}0^{t_2}(0^{m-1}1)^{\ell_1/m}0^{t_1}.$$

    We can continue this argument if $[n]_2$ has more than two runs of $1$'s and we find
    $T_m(n)$ is non-zero only when
    $$[n]_2 = 1^{\ell_s}0^{t_s}\cdots 1^{\ell_1}0^{t_1},\quad t_s \geq 1,\ldots,t_2 \geq 1, t_1 \geq 0,$$    
    for some $s$, and $\ell_i \equiv 0 \pmod{m}$ for $i=1,\ldots,s$.  Furthermore, for such $n$, we have $T_m(n)=1$, since the only $k$ that contributes a non-zero value to
    the sum defining $T_m(n)$ has the form
    $$[k]_2 = (0^{m-1}1)^{\ell_s/m}0^{t_s}\cdots(0^{m-1}1)^{\ell_1/m}0^{t_1}.$$

\end{proof}

The sequence $T_2$ can be viewed as a variant of the \emph{Baum--Sweet sequence} (OEIS~A086747)~\cite{BS76}, which is the sequence $(BS(n))_{n \geq 0}$
defined by
    \[
    BS(n) = 
        \begin{cases}
            1 & \text{if each run of 0's in } [n]_2 \text{ has even length};\\
            0 & \text{otherwise}.
        \end{cases}
    \]
In general then, the family of sequences $T_m$ could perhaps be taken as a family of
\emph{generalized Baum--Sweet sequences}.

\end{document}